\documentclass[letterpaper,10pt,conference]{ieeeconf}
\IEEEoverridecommandlockouts
\makeatletter

\let\proof\@undefined
\let\endproof\@undefined
\makeatother

\usepackage{graphicx,psfrag}
\usepackage[noend]{algorithmic}
\usepackage{xspace}
\usepackage{amssymb,amsmath,amscd,mathrsfs}
\usepackage{amsthm}
\usepackage{comment,color}
\usepackage{subfigure}
\usepackage{url}
\usepackage{multirow}

\usepackage[bookmarks=true]{hyperref}
\usepackage{fixltx2e}

\usepackage{color}
\usepackage{amsmath}
\usepackage{amssymb}
\usepackage{graphicx}
\usepackage{comment,xspace}
\usepackage{fancybox}

\newcommand{\real}{{\mathbb{R}}}
\newcommand{\reals}{\real}

\newtheorem{theorem}{Theorem}[section]

\newtheorem{remark}[theorem]{Remark}
\newtheorem{example}[theorem]{Example}
\newtheorem{definition}[theorem]{Definition}

\newcommand{\expectation}[1]{\mbox{$\mathbb{E}\left[#1\right]$}}

\newcommand{\fil}{\mathcal F}
\newcommand{\cs}{\mathcal Z}

\newcommand{\risk}{\rho}

\title{\LARGE \bf
Stochastic Optimal Control With \\Dynamic, Time-Consistent Risk Constraints
}

\author{Yin-Lam Chow, Marco Pavone
\thanks{Y.-L. Chow and M. Pavone are with the Department of Aeronautics and Astronautics, Stanford University, Stanford, CA 94305, USA. Email: {\tt \{ychow, pavone\}@stanford.edu}.}}

\begin{document}

\maketitle
\thispagestyle{empty}
\pagestyle{empty}

\begin{abstract}

In this paper we present a dynamic programing approach to stochastic optimal control problems with dynamic, time-consistent risk constraints. Constrained stochastic optimal control problems, which naturally arise when one has to consider multiple objectives, have been extensively investigated in the past 20 years; however, in most formulations, the constraints are formulated as either risk-neutral (i.e., by considering an expected cost), or by applying static, single-period risk metrics with limited attention to ``time-consistency" (i.e., to whether such metrics ensure rational consistency of risk preferences across multiple periods). Recently, significant strides have been made in the development of a rigorous theory of dynamic, \emph{time-consistent} risk metrics for multi-period (risk-sensitive) decision processes; however, their integration within constrained stochastic optimal control problems has received little attention. The goal of this paper is to bridge this gap. First, we formulate the stochastic optimal control problem with dynamic, time-consistent risk constraints and we characterize the tail subproblems (which requires the addition of a Markovian structure to the risk metrics). Second, we develop a dynamic programming approach for its solution, which allows to compute the optimal costs by value iteration. Finally, we discuss both theoretical and practical features of our approach, such as generalizations, construction of optimal control policies, and computational aspects. A simple, two-state example is given to illustrate the problem setup and the solution approach.
\end{abstract}

\section{Introduction}
Constrained stochastic optimal control problems naturally arise in several domains, including engineering, finance, and logistics. For example, in a telecommunication setting, one is often interested in the maximization of the throughput of some traffic subject to constraints on delays \cite{Altman_99, Korilis_95}, or seeks to minimize the average delays of some traffic types, while keeping the delays of other traffic types within a given bound \cite{Nain_86}. Arguably, the most common setup is the optimization of a \emph{risk-neutral expectation} criterion subject to a \emph{risk-neutral} constraint \cite{Chen_04, Piunovskiy_06, Chen_07}. This model, however, is not suitable in scenarios where risk-aversion is a key feature of the problem setup. For example, financial institutions are interested in trading assets while keeping the \emph{riskiness} of their portfolios below a threshold; or, in the optimization of rover planetary missions, one seeks to find a sequence of divert and driving maneuvers so that the rover drive is minimized and the \emph{risk} of a mission failure (e.g., due to a failed landing) is below a user-specified bound \cite{Pavone_12}. 

A common strategy to include risk-aversion in constrained problems is to have constraints where a static, single-period risk metric is applied to the future stream of costs; typical examples include variance-constrained stochastic optimal control problems (see, e.g., \cite{Piunovskiy_06, Sniedovich_80, Mannor_11}), or problems with probability constraints \cite{Chen_04, Piunovskiy_06}. However, using static, single-period risk metrics  in multi-period decision processes can lead to an over or under-estimation of the true dynamic risk, as well as to a potentially ``inconsistent" behavior (whereby risk preferences change in a seemingly irrational fashion between consecutive assessment periods), see \cite{Iancu_11} and references therein. In \cite{Rudloff_11}, the authors provide an example of a portfolio selection problem where the application of a static risk metric in a multi-period context leads a risk-averse decision maker to (erroneously) show risk neutral preferences at intermediate stages.

Indeed, in the recent past, the topic of \emph{time-consistent} risk assessment in multi-period decision processes has been heavily investigated \cite{rus_shapiro_2004, rus_shapiro_06, rus_shapiro_06_2, rus_09, shapiro_12, Cheridito_11, Penner_06}. The key idea behind time consistency is that if a certain outcome is considered less risky in all states of the world at stage $k$, then it should also be considered less risky at stage $k$ \cite{Iancu_11}. Remarkably, in \cite{rus_09}, it is proven that any risk measure that is time consistent can be represented as a composition of one-step conditional risk mappings, in other words, in multi-period settings, risk (as expected) should be compounded over time.

Despite the widespread usage of constrained stochastic optimal control and the significant strides in the theory of dynamic, time-consistent risk metrics, their integration within constrained stochastic optimal control problems has received little attention. The purpose of this paper is to bridge this gap. Specifically, the contribution of this paper is threefold. First, we formulate the stochastic optimal control problem with dynamic, time-consistent risk constraints and we characterize the tail subproblems (which requires the addition of a Markovian structure to the risk metrics). Second, we develop a dynamic programming approach for the solution, which allows to compute the optimal costs by value iteration. There are two main reasons behind our choice of a dynamic programing approach: (a) the dynamic programming approach can be used as an analytical tool in special cases and as the basis for the development of either exact or approximate solution algorithms; and (b) in the risk-neutral setting (i.e., both objective and constraints given as expectations of the sum of stage-wise costs) the dynamic programming approach appears numerical convenient with respect to other approaches (e.g., with respect to the convex analytic approach \cite{Altman_99}) and allows to build all (Markov) optimal control strategies \cite{Piunovskiy_06}. Finally, we discuss both theoretical and practical features of our approach, generalizations, construction of optimal control policies, and computational aspects. A simple, two-state example is given to illustrate the problem setup and the solution approach.

The rest of the paper is structured as follows. In Section \ref{sec:prelim} we present background material for this paper, in particular about dynamic, time-consistent risk measures. In Section \ref{sec:ps} we formally state the problem we wish to solve, while in Section \ref{sec:dp} we present a dynamic programming approach for the solution. In Section \ref{sec:dis} we discuss several aspects of our approach and provide a simple example. Finally, in Section \ref{sec:conc}, we draw our conclusions and offer directions for future work.

\section{Preliminaries}\label{sec:prelim}
In this section we provide some known concepts from the theory of Markov decision processes and of dynamic risk measures, on which we will rely extensively later in the paper.

\subsection{Markov Decision Processes}
A finite Markov Decision Process (MDP) is a four-tuple $(S, U, Q, U(\cdot))$, where $S$, the state space, is a finite set; $U$, the control space, is a finite set; for every $x\in S$, $U(x)\subseteq U$ is a nonempty set which represents the set of admissible controls when the system state is $x$; and, finally, $Q(\cdot|x,u)$ (the transition probability) is a conditional probability on $S$ given the set of admissible state-control pairs, i.e., the sets of pairs $(x,u)$ where $x\in S$ and $u\in U(x)$.

Define the space $H_k$ of admissible histories up to time $k$ by $H_k = H_{k-1} \times S\times U$, for $k\geq 1$, and $H_0=S$. A generic element $h_{0,k}\in H_k$ is of the form $h_{0,k} = (x_0, u_0, \ldots , x_{k-1}, u_{k-1}, x_k)$. Let $\Pi$ be the set of all deterministic policies with the property that at each time $k$ the control is a function of $h_{0,k}$. In other words, $\Pi := \Bigl \{ \{\pi_0: H_0 \rightarrow U,\, \pi_1: H_1 \rightarrow U, \ldots\} | \pi_k(h_{0,k}) \in U(x_k) \text{ for all } h_{0,k}\in H_k, \, k\geq 0 \Bigr\}$.

\subsection{Time-Consistent Dynamic Risk Measures}
This subsection follows closely the discussion in \cite{rus_09}. Consider a probability space $(\Omega, \fil, P)$, a filtration $\fil_1\subset \fil_2 \cdots \subset \fil_N \subset \fil$, and an adapted sequence of random variables $Z_k$, $k\in \{0, \cdots,N\}$. We assume that $\fil_0 = \{\Omega, \emptyset\}$, i.e., $Z_0$ is deterministic. In this paper we interpret the variables $Z_k$ as stage-wise costs. For each $k\in\{1, \cdots, N\}$, define the spaces of random variables with finite $p$th order moment as $\cs_k:=  L_p(\Omega, \fil_k, P)$, $p\in [1,\infty]$; also, let $\cs_{k, N}:=\cs_k \times \cdots \times \cs_N$.

The fundamental question in the theory of dynamic risk measures is the following: how do we evaluate the risk of the subsequence $Z_k, \ldots, Z_N$ from the perspective of
stage $k$? Accordingly, the following definition introduces the concept of dynamic risk measure (here and in the remainder of the paper equalities and inequalities are in the almost sure sense).

\begin{definition}[Dynamic Risk Measure]
A dynamic risk measure is a sequence of mappings  $\risk_{k,N}:\cs_{k, N}\rightarrow\cs_k$, $k\in\{0, \ldots,N\}$, obeying the following monotonicity property:
\[
\risk_{k,N}(Z)\leq \risk_{k,N}(W)   \text{ for all } Z,W \in\cs_{k,N} \text{ such that } Z\leq W.
\]
\end{definition}
The above monotonicity property is arguably a natural requirement for any meaningful dynamic risk measure. Yet, it does not imply the following notion of \emph{time consistency}:
\begin{definition}[Time Consistency]
A dynamic risk measure $\{ \risk_{k,N}\}_{k=0}^N$ is called time-consistent if, for all $0\leq l<k\leq N$ and all sequences $Z, W \in \cs_{l,N}$, the conditions
\begin{equation}
\begin{split}
&Z_i = W_i,\,\, i = l,\cdots,k-1, \text{ and }\\
&\risk_{k,N}(Z_k, \cdots,Z_N)\leq \risk_{k,N}(W_k, \cdots,W_N),
\end{split}
\end{equation}
imply that
\[
 \risk_{l,N}(Z_l, \cdots,Z_N)\leq \risk_{l,N}(W_l, \cdots,W_N).
\]
\end{definition}
In other words, if the $Z$ cost sequence is deemed less risky than the $W$ cost sequence from the perspective of a future time $k$, and they yield identical costs from the current time $l$ to the future time $k$, then the $Z$ sequence should be deemed as less risky at the current time $l$, as well. The pitfalls of time-inconsistent dynamic risk measures have already been mentioned in the introduction and are discussed in detail in  \cite{Cheridito_09, shapiro_09, Iancu_11}.

The issue then is what additional ``structural" properties are required for a dynamic risk measure to be time consistent. To answer this question we need one more definition:
\begin{definition}[Coherent one-step conditional risk measures]
A coherent one-step conditional risk measures is a mapping $\risk_k:\cs_{k+1}\rightarrow \cs_k$, $k\in\{0,\ldots,N\}$, with the following four properties:
\begin{itemize}
\item Convexity: $\risk_k(\lambda Z + (1-\lambda)W)\leq \lambda\risk_k(Z) + (1-\lambda)\risk_k(W)$, $\forall \lambda\in[0,1]$ and $Z,W \in\cs_{k+1}$;
\item Monotonicity:  if $Z\leq W$ then $\risk_k(Z)\leq\risk_k(W)$, $\forall Z,W \in\cs_{k+1}$;
\item Translation invariance:  $\risk_k(Z+W)=Z + \risk_k(W)$, $\forall Z\in\cs_k$ and $W \in \cs_{k+1}$;
\item Positive homogeneity: $\risk_k(\lambda Z) = \lambda \risk_k(Z)$, $\forall Z \in \cs_{k+1}$ and $\lambda\geq 0$.
\end{itemize}
\end{definition}

We are now in a position to state the main result of this section.
\begin{theorem}[Dynamic, time-consistent risk measures]\label{thrm:tcc}
Consider, for each $k\in\{0,\cdots,N\}$, the mappings $\risk_{k,N}:\cs_{k, N}\rightarrow\cs_k$ defined as
\begin{equation}\label{eq:tcrisk}
\begin{split}
\risk_{k,N} &= Z_k + \risk_k(Z_{k+1} + \risk_{k+1}(Z_{k+2}+\ldots+\\
	&\qquad\risk_{N-2}(Z_{N-1}+\risk_{N-1}(Z_N))\ldots)),
\end{split}
\end{equation}
where the $\risk_k$'s are coherent one-step risk measures. Then, the ensemble of such mappings is a time-consistent dynamic risk measure.
\end{theorem}
\begin{proof}
See \cite{rus_09}.
\end{proof}
Remarkably, Theorem 1 in \cite{rus_09} shows (under weak assumptions) that the ``multi-stage composition" in equation \eqref{eq:tcrisk} is indeed \emph{necessary for time consistency}. Accordingly, in the remainder of this paper, we will focus on the \emph{dynamic, time-consistent risk measures} characterized in Theorem \ref{thrm:tcc}.

With dynamic, time-consistent risk measures, since at stage $k$ the value of $\risk_k$ is $\fil_k$-measurable, the evaluation of risk can depend on the whole past (even though in a time-consistent way). On the one hand, this generality appears to be of little value in most practical cases, on the other hand, it leads to optimization problems that are intractable from a computational standpoint (and, in particular, do not allow for a dynamic programing solution). For these reasons, in this paper we consider a (slight) refinement of the concept of dynamic, time-consistent risk measure, which involves the addition of a Markovian structure \cite{rus_09}.  
\begin{definition}[Markov dynamic risk measures]\label{def:Markov}
Let $\mathcal V:=L_p(S, \mathcal B, P)$ be the space of random variables on $S$ with finite $p$\emph{th} moment. Given a controlled Markov process $\{x_k\}$, a Markov dynamic risk measure is a dynamic, time-consistent risk measure if each coherent one-step risk measure $\risk_k:\cs_{k+1}\rightarrow \cs_k$ in equation \eqref{eq:tcrisk} can be written as:
\begin{equation}\label{eq:Markov}
\risk_k(V(x_{k+1})) = \sigma_k(V(x_{k+1}),x_k, Q(x_{k+1} |x_k, u_k)),
\end{equation}
for all $V(x_{k+1})\in \mathcal V$ and $u\in U(x_k)$, where $\sigma_k$ is a coherent one-step risk measure on $\mathcal V$ (with the additional technical property that for every $V(x_{k+1})\in \mathcal V$ and $u\in U(x_k)$ the function $x_k \mapsto \sigma_k(V(x_{k+1}), x_k, Q(x_{k+1}|x_k, u_k))$ is an element of $\mathcal V$). 
\end{definition}

In other words, in a Markov dynamic risk measures, the evaluation of risk is not allowed to depend on the whole past.

\begin{example} An important example of coherent one-step risk measure satisfying the requirements for Markov dynamic risk measures (Definition \ref{def:Markov}) is the mean-semideviation risk function:
\begin{align}
\risk_k(V) = \expectation{V} + \lambda\, \Bigl (\expectation{[V - \expectation{V}]_{+}^p} \Bigr)^{1/p},\label{MD_risk}
\end{align}
where $p\in[1,\infty)$, $[z]_+^p:=(\max(z,0))^p$, and $\lambda\in [0,\, 1].$
\end{example}
Other important examples include the conditional average value at risk and, of course, the risk-neutral expectation \cite{rus_09}.
Accordingly, in the remainder of this paper we will restrict our analysis to Markov dynamic risk measures.

\section{Problem Statement}\label{sec:ps}
In this section we formally state the problem we wish to solve. Consider an MDP and let $c : S \times U \rightarrow \reals$ and $d : S \times U \rightarrow \reals$ be functions which denote costs associated with state-action pairs. Given a policy $\pi\in \Pi$, an initial state $x_0\in S$, and an horizon $N\geq 1$, the cost function is defined as
\[
J^{\pi}_N(x_0):=\expectation{\sum_{k=0}^{N-1}\, c(x_k, u_k)},
\]
and the risk constraint is defined as
\[
R^{\pi}_N(x_0):= \risk_{0,N}\Bigl(d(x_0,u_0), \ldots, d(x_{N-1},u_{N-1}),0\Bigr),
\]
where $\risk_{k,N}(\cdot)$, $k\in \{0,\ldots, N-1\}$, is a time consistent multi-period risk measure with $\risk_i$ being a Markov risk measure for any $i\in[k,N-1]$ (for simplicity, we do not consider terminal costs, even though their inclusion is straightforward).  The problem we wish to solve is then as follows:
\begin{quote} {\bf Optimization problem $\mathcal{OPT}$} --- Given an initial state $x_0\in S$, a time horizon $N\geq 1$, and a risk threshold $r_0 \in \reals$, solve
\begin{alignat*}{2}
\min_{\pi \in \Pi} & & \quad&J^{\pi}_N(x_0) \\  
\text{subject to} & & \quad&R^{\pi}_N(x_0) \leq r_0.
\end{alignat*}
\end{quote}
If problem $\mathcal OPT$ is not feasible, we say that its value is $\overline C$, where $\overline C$ is a ``large" constant (namely, an upper bound over the $N$-stage cost). Note that, when the problem is feasible, an optimal policy always exists since the state and control spaces are finite. When $\risk_{0,N}$ is replaced by an expectation, we recover the usual risk-neutral constrained stochastic optimal control problem studied, e.g., in \cite{Chen_04, Piunovskiy_06}. In the next section we present a dynamic programing approach to solve problem $\mathcal OPT$.

\section{A Dynamic Programming Algorithm for Risk-Constrained Multi-Stage Decision-Making}\label{sec:dp}
In this section we discuss a dynamic programming approach to solve problem $\mathcal{OPT}$. We first characterize the relevant value functions, and then we present the Bellman's equation that such value functions have to satisfy.

\subsection{Value functions}
Before defining the value functions we need to define the tail subproblems. For a given $k\in \{0,\ldots,N-1\}$ and a given state $x_k\in S$, we define the \emph{sub-histories} as $h_{k,j}:=(x_k, u_k, \ldots,x_j)$ for $j\in \{k,\ldots, N\}$; also, we define the \emph{space of truncated policies} as $\Pi_k:=\Bigl \{ \{\pi_k, \pi_{k+1}, \ldots\} | \pi_j(h_{k,j}) \in U(x_j) \text{ for } j\geq k \Bigr\}$. For a given stage $k$ and state $x_k$, the cost of the tail process associated with a policy $\pi\in \Pi_k$ is simply $J^{\pi}_N(x_k):=\expectation{\sum_{j=k}^{N-1}\, c(x_j, u_j)}$. The risk associated with the tail process is:
\[
R^{\pi}_N(x_k):= \risk_{k,N}\Bigl(d(x_k,u_k), \ldots, d(x_{N-1},u_{N-1}),0\Bigr),
\]
which is \emph{only} a function of the current state $x_k$ and does \emph{not} depend on the history $h_{0,k}$ that led to $x_k$. This crucial fact stems from the assumption that $\{\risk_{k,N}\}_{k=0}^{N-1}$ is a Markov dynamic risk measure, and hence the evaluation of risk only depends on the \emph{future} process and on the present state $x_k$ (formally, this can be easily proven by repeatedly applying equation \eqref{eq:Markov}). Hence, the tail subproblems are \emph{completely} specified by the knowledge of $x_k$ and are defined as
\begin{alignat}{2}
\min_{\pi \in \Pi_k} & & \quad&J^{\pi}_N(x_k) \label{problem_SOCP}\\
\text{subject to} & & \quad&R^{\pi}_N(x_k) \leq r_k(x_k),\label{constraint_SOCP}
\end{alignat}
for a given (undetermined) threshold value $r_k(x_k) \in \reals$ (i.e., the tail subproblems are specified up to a threshold value). We are interested in characterizing a ``minimal" set of \emph{feasible} thresholds at each step $k$, i.e., a ``minimal" interval of thresholds for which the subproblems are feasible.

The minimum risk-to-go for each state $x_k\in S$ and $k\in \{0,\ldots,N-1\}$ is given by:
\[
\underline{R}_N(x_k):=\min_{\pi\in \Pi_k} \, R_N^{\pi}(x_k).
\]
Since $\{\risk_{k,N}\}_{k=0}^{N-1}$ is a Markov dynamic risk measure, $\underline{R}_N(x)$ can be computed by using a dynamic programming recursion (see Theorem 2 in \cite{rus_09}). The function $\underline{R}_N(x_k)$ is clearly the lowest value for a feasible constraint threshold. To characterize the upper bound, let:
\[
\risk_{\text{max}}:=\max_{k\in\{0,\ldots,N-1\}} \, \, \max_{(x,u)\in S\times U} \, \risk_k(d(x,u)).
\]
By the monotonicity and translation invariance of Markov dynamic risk measures, one can easily show that
\[
\max_{\pi\in \Pi_k} \, R_N^{\pi}(x_k)\leq (N-k)\risk_{\text{max}}:=\overline{R}_N.
\]
Accordingly, for each $k\in\{0,\ldots, N-1\}$ and $x_k\in S$, we define the set of feasible constraint thresholds as 
\[
\Phi_k(x_k):=[\underline{R}_N(x_k), \overline{R}_N],\quad \Phi_N(x_N):=\{0\}.
\]
(Indeed, thresholds larger than $\overline{R}_N$ would still be feasible, but would be redundant and would increase the complexity of the optimization problem.)

The value functions are then defined as follows:
\begin{itemize}
\item If $k<N$ and $r_k \in \Phi_k(x_k)$:
\begin{alignat*}{2}
V_k(x_k, r_k)  = &\min_{\pi \in \Pi_k} & \quad&J^{\pi}_N(x_k) \\  
&\text{subject to} & &R^{\pi}_N(x_k) \leq r_k(x_k);
\end{alignat*}
the minimum is well-defined since the state and control spaces are finite.
\item iI $k\leq N$ and $r_k \notin \Phi_k(x_k)$:
\[
V_k(x_k, r_k)  = \overline{C};
\]
\item when $k=N$ and $r_N=0$:
\[
V_N(x_N,r_N) = 0.
\]
\end{itemize}

Clearly, for $k=0$, we have the definition of problem $\mathcal OPT$.

\subsection{Dynamic programming recursion}
In this section we prove that the value function can be computed by dynamic programming. Let $B(S)$ denote the space of real-valued bounded functions on $S$, and $B(S \times \reals)$ denote the space of real-valued bounded functions on $S\times \reals$. For $k\in \{0, \ldots, N-1\}$, we define the dynamic programming operator $T_k[V_k] : B(S \times \reals) \mapsto B(S \times \reals)$ according to the equation:
\begin{equation}\label{eq:T}
\begin{split}
T_k[V_{k+1}]&(x_k, r_k) := \inf_{(u,r^{\prime})\in F_k(x_k, r_k)} \, \biggl\{c(x_k,u) \, \,+\\
& \ \sum_{x_{k+1}  \in S} \, Q(x_{k+1}|x_k,u)\, V_{k+1}(x_{k+1}, r^{\prime}(x_{k+1})) \biggr\},
\end{split}
\end{equation}
where $F_k \subset \reals \times B(S)$ is the set of control/threshold \emph{functions}:
\begin{equation*}
\begin{split}
F_k(x_k,& r_k):= \biggr\{(u, r^{\prime}) \Big | u\in U(x_k), r^{\prime}(x^{\prime}) \in \Phi_{k+1}(x^{\prime}) \text{ for}\\& \text{all } x^{\prime} \in S, \text{ and } d(x_k, u) + \risk_k(r^{\prime}(x_{k+1}))  \leq r_k\biggl\}.
\end{split}
\end{equation*}
If $F_k(x_k,r_k) = \emptyset$, then $T_k[V_{k+1}](x_k, r_k)=\overline C$.

Note that, for a given state and threshold constraint, set $F_k$ characterizes the set of feasible pairs of actions and subsequent constraint
thresholds. Feasible subsequent constraint thresholds are thresholds which if satisfied at the next stage ensure that the current
state satisfies the given threshold constraint (see \cite{Chen_07} for a similar statement in the risk-neutral case). Also, note that equation \eqref{eq:T} involves a functional minimization over the Banach space $B(S)$. Indeed, since $S$ is finite, $B(S)$ is isomorphic with $\reals^{|S|}$, hence the minimization in equation \eqref{eq:T} can be re-casted as a regular (although possibly large) optimization problem in the Euclidean space. Computational aspects are further discussed in the next section. 

We are now in a position to prove the main result of this paper.
\begin{theorem}[Bellman's equation with risk constraints]\label{TC_good}
Assume that the infimum in equation \eqref{eq:T} is attained. Then, for all $k\in \{0, \ldots, N-1\}$, the optimal cost functions satisfy the Bellman's equation:
\[
V_k(x_k, r_k) = T_k[V_{k+1}](x_k, r_k).
\]
\end{theorem}
\begin{proof}
The proof style is similar to that of Theorem 3.1 in \cite{Chen_04}. The proof consists of two steps. First, we show that $V_k(x_k, r_k) \geq T_k[V_{k+1}](x_k, r_k)$ for all pairs $(x_k, r_k) \in S\times \reals$. Second, we show  $V_k(x_k, r_k) \leq T_k[V_{k+1}](x_k, r_k)$ for all pairs $(x_k, r_k) \in S\times \reals$. These two results will prove the claim.

\emph{Step (1)}. If $r_k \notin \Phi_k(x_k)$, then, by definition, $V_k(x_k, r_k) = \overline{C}$. Also, $r_k \notin \Phi_k(x_k)$ implies that $F_k(x_k, r_k)$ is empty (this can be easily proven by contradiction). Hence, $T_k[V_{k+1}](x_k, r_k) = \overline C$. Therefore, if $r_k \notin \Phi_k(x_k)$,
\begin{equation}\label{eq:easy}
V_k(x_k, r_k) =  \overline{C} = T_k[V_{k+1}](x_k, r_k),
\end{equation}
i.e., $V_k(x_k, r_k)\geq T_k[V_{k+1}](x_k, r_k)$.

Assume, now, $r_k \in \Phi_k(x_k)$. Let $\pi^* \in \Pi_k$ be the optimal policy that yields the optimal cost $V_k(x_k, r_k)$. Construct the ``truncated" policy $\bar \pi \in \Pi_{k+1}$ according to:
\[
\bar \pi_j(h_{k+1,j}):= \pi^*_j(x_k, \pi_k^*(x_k), h_{k+1,j}), \quad \text{ for } j\geq k+1.
\]
In other words, $\bar \pi$ is a policy in $\Pi_{k+1}$ that acts as prescribed by $\pi^*$. By applying the law of total expectation, we can write:
\begin{equation*}
\begin{split}
&V_k(x_k, r_k) = \expectation{\sum_{j=k}^{N-1} \, c(x_j, \pi^*_j(h_{k,j}))} = \\
&\quad c(x_k, \pi_k^*(x_k)) +  \expectation{\sum_{j=k+1}^{N-1} \, c(x_j, \pi^*_j(h_{k,j}))}=\\
&\quad c(x_k, \pi_k^*(x_k)) +   \expectation{\expectation{\sum_{j=k+1}^{N-1} \, c(x_j, \pi_j^*(h_{k,j})) \, \Big | \, h_{k,k+1}}}.
\end{split}
\end{equation*}
Note that $\expectation{\sum_{j=k+1}^{N-1} \, c(x_j, \pi^*_j(h_{k,j})) \, \Big | \, h_{k,k+1}} = J_N^{\bar \pi}(x_{k+1})$.
Clearly, the truncated policy $\bar \pi$ is a feasible policy for the tail subproblem 
\begin{alignat*}{2}
\min_{\pi \in \Pi_{k+1}} & & \quad&J^{\pi}_N(x_{k+1}) \\  
\text{subject to} & & \quad&R^{\pi}_N(x_{k+1}) \leq R^{\bar \pi}_N(x_{k+1}).
\end{alignat*}
Collecting the above results, we can write 
\begin{equation*}
\begin{split}
V_k(x_k, r_k)  &= c(x_k, \pi_k^*(x_k))  +  \expectation{J_N^{\bar \pi}(x_{k+1})} \\
&\geq c(x_k, \pi_k^*(x_k))  +  V_{k+1}(x_{k+1}, R^{\bar \pi}_N(x_{k+1}))\\
&\geq T_k[V_{k+1}](x_k, r_k),
\end{split}
\end{equation*}
where the last inequality follows from the fact that $R^{\bar \pi}_N(\cdot)$ can be viewed as a valid threshold function in the minimization in equation \eqref{eq:T}.

\emph{Step (2)}.  If $r_k \notin \Phi_k(x_k)$, equation \eqref{eq:easy} holds and, therefore, $V_k(x_k, r_k)\leq T_k[V_{k+1}](x_k, r_k)$.

Assume $r_k \notin \Phi_k(x_k)$. For a given pair $(x_k, r_k)$, where $r_k \in \Phi_k(x_k)$, let $u^*$ and $r^{\prime, *}$ the minimizers in equation \eqref{eq:T} (here we are exploiting the assumption that the minimization problem in equation \eqref{eq:T} admits a minimizer). By definition, $r^{\prime, *}(x_{k+1}) \in \Phi_{k+1}(x_{k+1})$ for all $x_{k+1}\in S$. Also, let $\pi^* \in \Pi_{k+1}$ the optimal policy for the tail subproblem:
\begin{alignat*}{2}
\min_{\pi \in \Pi_{k+1}} & & \quad&J^{\pi}_N(x_{k+1}) \\  
\text{subject to} & & \quad&R^{\pi}_N(x_{k+1}) \leq r^{\prime,*}(x_{k+1}).
\end{alignat*}
Construct  the ``extended" policy $\bar \pi \in \Pi_k$ as follows:
\[
\bar \pi_k(x_k) = u^*, \text{ and } \bar \pi_j(h_{k,j}) = \pi^*_j(h_{k+1,j}) \text{ for } j\geq k+1.
\]
Since $\pi^*$ is an optimal, and a fortiori feasible, policy for the tail subproblem (from stage $k+1$) with threshold function $r^{\prime, *}$, the policy $\bar \pi \in \Pi_k$ is a feasible policy for the tail subproblem (from stage $k$):
\begin{alignat*}{2}
\min_{\pi \in \Pi_{k}} & & \quad&J^{\pi}_N(x_{k}) \\  
\text{subject to} & & \quad&R^{\pi}_N(x_{k}) \leq r_k.
\end{alignat*}
Hence, we can write
\begin{equation*}
\begin{split}
&V_{k}(x_k, r_k) \leq J^{\bar \pi}_N(x_k) = \\
&\quad c(x_k, \bar \pi_k(x_k)) + \expectation{\expectation{\sum_{j=k+1}^{N-1} \, c(x_j, \bar \pi_j(h_{k,j})) \, \Big | \, h_{k,k+1}}}.
\end{split}
\end{equation*}
Note that $\expectation{\sum_{j=k+1}^{N-1} \, c(x_j, \bar \pi_j(h_{k,j})) \, \Big | \, h_{k,k+1}} = J_N^{\pi^*}(x_{k+1})$. Hence, from the definition of $\pi^*$, one easily obtains:
\begin{equation*}
\begin{split}
&V_{k}(x_k, r_k) \leq c(x_k, \bar \pi_k(x_k)) +  \expectation{J^{\pi^*}_N(x_{k+1})} =  c(x_k, u^*) \, +\\
&\qquad  \sum_{x_{k+1}\in S}  Q(x_{k+1}|x_k, u^*) V_{k+1}(x_{k+1}, r^{\prime, *}(x_{k+1}))=
\\
&\qquad T_k[V_{k+1}](x_k, r_k). 
\end{split}
\end{equation*}

Collecting the above results, the claim follows.
\end{proof}

\begin{remark}[On the assumption in Theorem \ref{TC_good}] In Theorem \ref{TC_good} we assume that the infimum in equation \eqref{eq:T} is attained. This is indeed true under very weak conditions (namely that $U(x_k)$ is a compact set, $\sigma_k(\nu(x_{k+1}),x_{k+1},Q)$ is a lower semi-continuous function in $Q$, $Q(x_k,u_k)$ is continuous in $u_k$ and the stage-wise cost $c$ and $d$ are lower semi-continuous in $u_k$). The proof of this statement is omitted in the interest of brevity and is left for a forthcoming publication.
%
\end{remark}

\section{Discussion}\label{sec:dis}
In this section we show how to construct optimal policies, discuss computational aspects, and present a simple two-state example for machine repairing.

\subsection{Construction of optimal policies}
Under the assumption of Theorem \ref{TC_good}, optimal control policies can be constructed as follows. For any given $x_k\in S$ and $r_k \in \Phi_k(x_k)$, let $u^*(x_k, r_k)$ and $r^{\prime}(x_k, r_k)(\cdot)$ be the minimizers in equation \eqref{eq:T} (recall that $r^{\prime}$ is a function). 
\begin{theorem}[Optimal policies]\label{them:optPoli}
Assume that the infimum in equation \eqref{eq:T} is attained. Let $\pi \in \Pi$ be a policy recursively defined as follows:
\begin{equation*}
\begin{split}
\pi_k(h_{0,k}) = u^*(x_k, r_k) \text{ \em with  } r_k = r^{\prime}(x_{k-1},r_{k-1})(x_k),
\end{split}
\end{equation*}
when $k\in \{1,\ldots, N-1\}$, and
\[
\pi(x_0) = u^*(x_0, r_0),
\]
for a given threshold $r_0\in \Phi_0(x_0)$. Then, $\pi$ is an optimal policy for problem $\mathcal{OPT}$ with initial condition $x_0$ and constraint threshold $r_0$.
\end{theorem}
\begin{proof}
As usual for dynamic programming problems, the proof uses induction arguments (see, in particular, \cite{bertsekas_05} and \cite[Theorem 4]{Chen_07} for a similar proof in the risk-neutral case).

Consider a tail subproblem starting at stage $k$, for $k=0,\ldots,N-1$; for a given initial state $x_k\in S$ and constraint threshold $r_k \in \Phi_k(x_k)$, let $\pi^{k, r_k} \in \Pi_k$ be a policy recursively defined as follows:
\[
\pi_j^{k, r_k}(h_{k,j}) = u^*(x_j, r_j)  \text{ with  } r_j=r^{\prime}(x_{j-1}, r_{j-1})(x_j),
\]
when $j\in \{k+1,\ldots, N-1\}$, and
\[
\pi_k^{k, r_k}(x_k) = u^*(x_k, r_k).
\]
We prove by induction that $\pi^{k, r_k}$ is optimal. Clearly, for $k=0$, such result implies the claim of the theorem.

Let $k=N-1$ (base case). In this case the tail subproblem is:
\begin{alignat*}{2}
\min_{\pi \in \Pi_{N-1}} & & \quad&c(x_{N-1}, \pi(x_{N-1})) \\  
\text{subject to} & & \quad& d(x_{N-1}, \pi(x_{N-1})) \leq r_{N-1}.
\end{alignat*}
Since, by definition, $r^{\prime}(x_N)$ and $V_N(x_N, r_N)$ are identically equal to zero, and due to the positive homogeneity of one-step conditional risk measures, the above tail subproblem is identical to the optimization problem in the Bellman's recursion \eqref{eq:T}, hence $\pi^{N-1, r_{N-1}}$ is optimal.

Assume as induction step that $\pi^{k+1,r_{k+1}}$ is optimal for the tail subproblems starting at stage $k+1$ with $x_{k+1}\in S$ and $r_{k+1} \in \Phi_{k+1}(x_{k+1})$. We want to prove that $\pi^{k, r_k}$ is optimal for the tail subproblems starting at stage $k$ with initial state $x_k\in S$ and constraint threshold $r_k \in \Phi_k(x_k)$. First, we prove that $\pi^{k, r_k}$ is a feasible control policy. Note that, from the recursive definitions of $\pi^{k, r_k}$ and $\pi^{k+1, r_{k+1}}$, one has
\[
R_N^{\pi^{k, r_k}}(x_{k+1}) = R_N^{\pi^{k+1, r^{\prime}(x_k, r_k)(x_{k+1})}}(x_{k+1}).
\]
Hence, one can write:
\begin{equation}
\begin{split}
&R_N^{\pi^{k, r_k}}(x_k) = d(x_k, u^*(x_k, r_k)) + \risk_k(R_N^{\pi^{k, r_k}}(x_{k+1}))=\\
&\qquad d(x_k, u^*(x_k, r_k)) + \risk_k\Bigl(R_N^{\pi^{k+1, r^{\prime}(x_k, r_k)(x_{k+1})}}(x_{k+1})\Bigr)\leq\\
&\qquad d(x_k, u^*(x_k, r_k)) + \risk_k\Bigl(r^{\prime}(x_k, r_k)(x_{k+1})\Bigr) \leq r_k,
\end{split}
\end{equation}
where the first inequality follows from the inductive step and the monotonicity of coherent one-step conditional risk measures, and the last step follows from the definition of $u^*$ and $r^{\prime}$. Hence, $\pi^{k, r_k}$ is a feasible control policy (assuming initial state $x_k\in S$ and constraint threshold $r_k \in \Phi_k(x_k)$). As for its cost, one has, similarly as before,
\[
J_N^{\pi^{k, r_k}}(x_{k+1}) = J_N^{\pi^{k+1, r^{\prime}(x_k, r_k)(x_{k+1})}}(x_{k+1}).
\]
Then, one can write:
\begin{equation}
\begin{split}
&J_N^{\pi^{k, r_k}}(x_k) = c(x_k, u^*(x_k, r_k)) + \expectation{J_N^{\pi^{k, r_k}}(x_{k+1})}=\\
&\qquad c(x_k, u^*(x_k, r_k)) + \expectation{ J_N^{\pi^{k+1, r^{\prime}(x_k, r_k)(x_{k+1})}}(x_{k+1})}=\\
&\qquad c(x_k, u^*(x_k, r_k)) +\expectation{V_{k+1}(x_{k+1}),  r^{\prime}(x_k, r_k)(x_{k+1})}=\\
&\qquad T_k[V_{k+1}](x_k, r_k) = V_k(x_k, r_k),
\end{split}
\end{equation}
where the third equality follows from the inductive step, the fourth equality follows form the definition of the dynamic programming operator in equation \eqref{eq:T}, and the last equality follows from Theorem \ref{TC_good}. Since policy $\pi^{k, r_k}$ is feasible and achieves the optimal cost, it is optimal. This concludes the proof.
\end{proof}

Note that the optimal policy in the statement of Theorem \ref{them:optPoli} can be written in ``compact" form without the aid of the extra variable $r_k$. Indeed, for $k=1$, by defining the threshold transition function $\mathcal R_1(h_{0,1}):=r'(x_{0},r_{0})(x_1)$, one can write $r_1=\mathcal{R}_1(h_{0,1})$. Then, by induction arguments, one can write, for any $k\in\{1,\ldots,N\}$, $r_{k}=\mathcal{R}_k(h_{0,k})$, where $\mathcal R_k$ is the threshold transition function at stage $k$. Therefore, the optimal policy in the statement of Theorem \ref{them:optPoli} can be written as $\pi(h_{0,k}) = u^*(x_k, \mathcal R_k(h_{0,k}))$, which makes explicit the dependency of $\pi$ over the process history.

Interestingly, if one views the constraint thresholds as state variables, the optimal policies of problem $\mathcal{OPT}$ have a Markovian structure with respect to the augmented control problem. 

\subsection{Computational issues}
In our approach, the solution of problem $\mathcal{OPT}$ entails the solution of two dynamic programing problems, the first one to find the lower bound for the set of feasible constraint thresholds (i.e., the function $\underline{R}(x)$, see Section \ref{sec:dp}), and the second one to compute the value functions $V_k(x_k, r_k)$. The latter problem is the most challenging one since it involves a functional minimization. However, as already noted, since $S$ is finite, $B(S)$ is isomorphic with $\mathbb{R}^{|S|}$, and the functional minimization in the Bellman operator \eqref{eq:T} can be re-casted as an optimization problem in the Euclidean space. This problem, however, can be large and, in general, is not convex. 

\begin{figure}[h]
  \centering
  {
  \includegraphics[width = 0.23\textwidth]{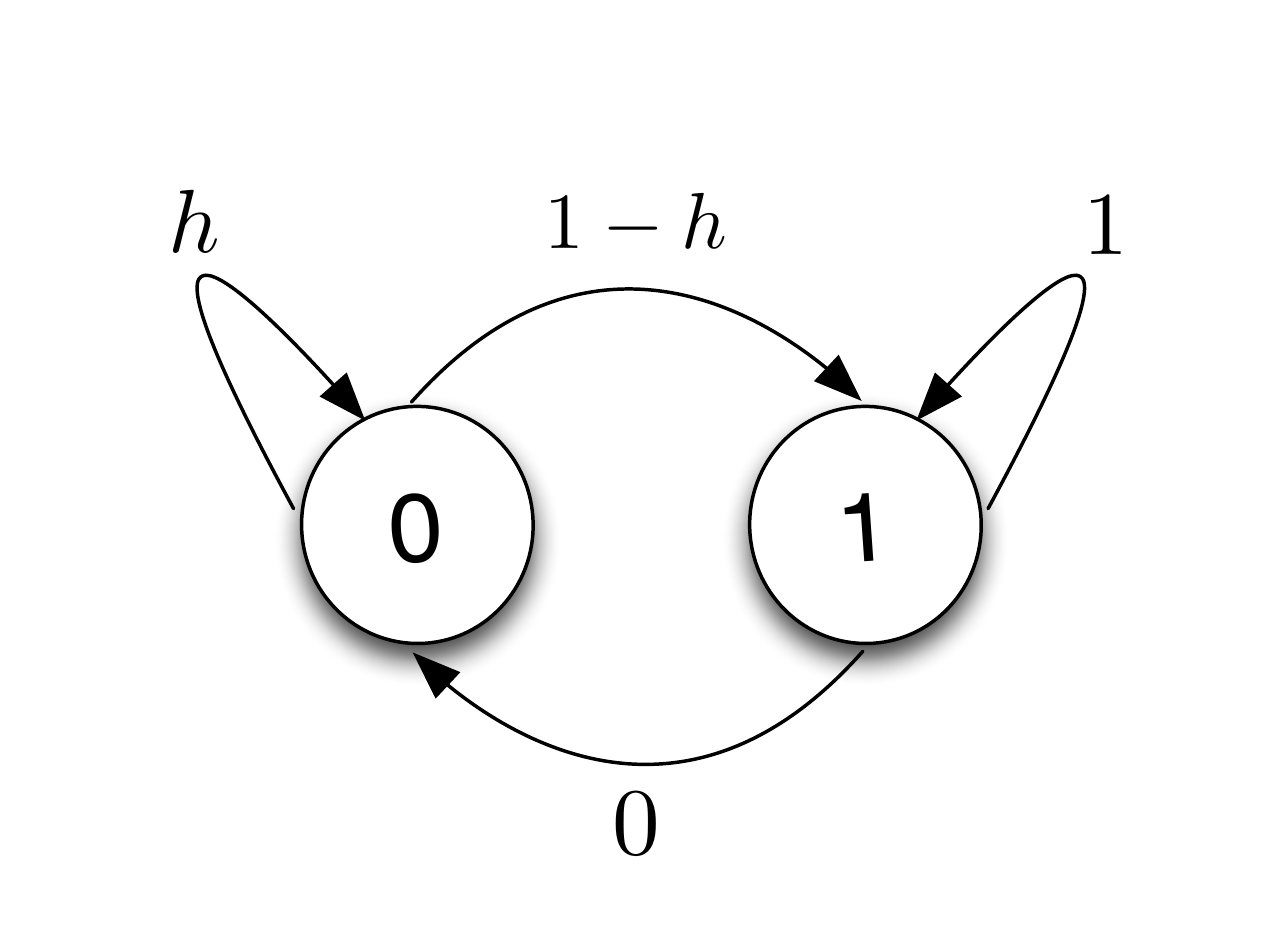}
  \includegraphics[width = 0.23\textwidth]{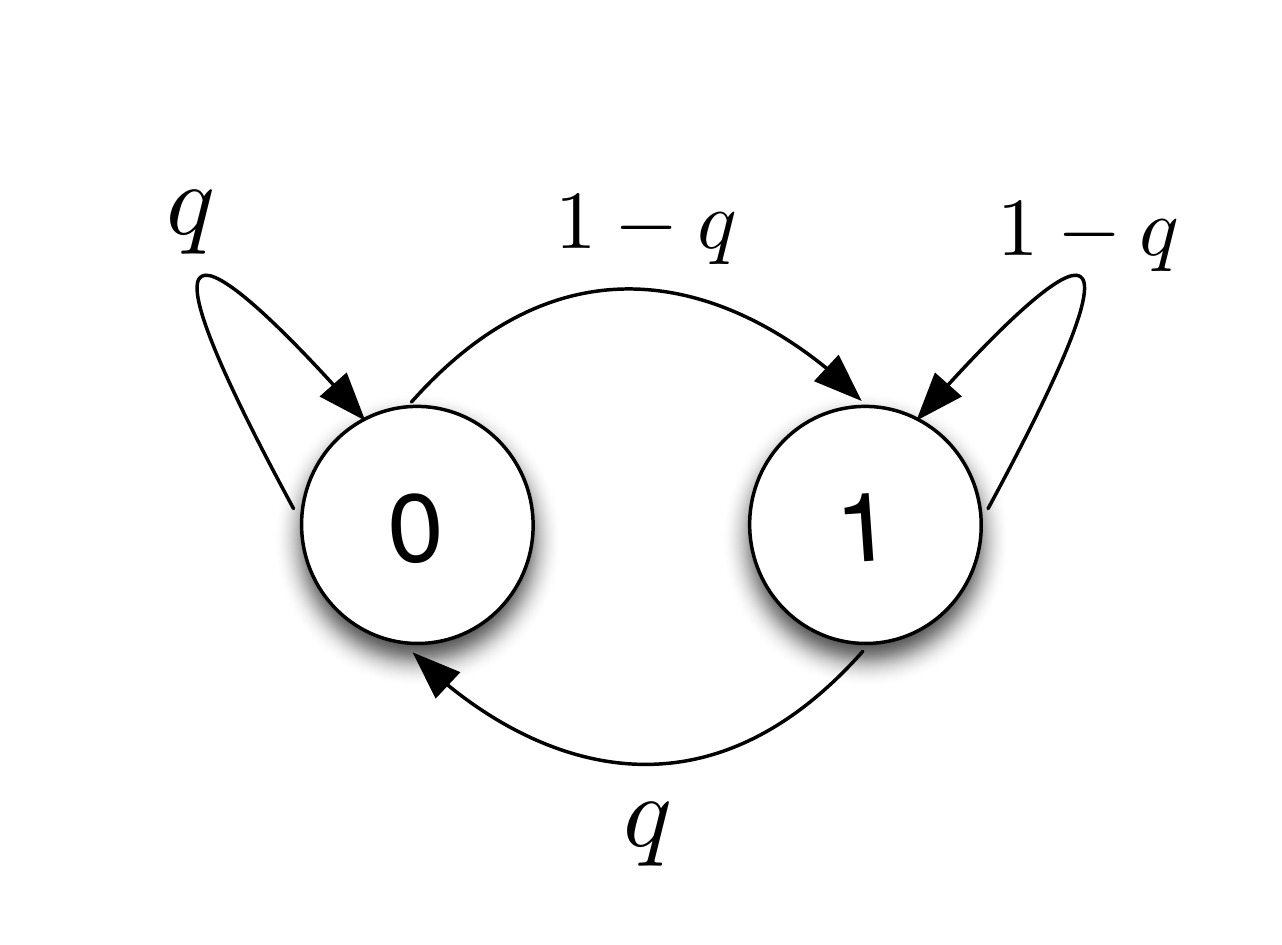}
}
  \caption{Left figure: transition probabilities for control $u=0$. Right figure: transition probabilities for control $u=1$. Circles represent states. The transition probabilities satisfy $1\geq q>h\geq 0$.}
  \label{fig:tp}
\end{figure}

\subsection{System maintenance example}
Finally, we illustrate the above concepts with a simple two-stage (i.e., $N=2$) example that represents the problem of scheduling maintenance operations for a given system. The state space is given by $S=\{0,1\}$, where $\{0\}$ represents a normal state and $\{1\}$ represents a failure state; the control space is given by $U=\{0,1\}$, where $\{0\}$ means ``do nothing" and $\{1\}$ means ``perform maintenance". The transition probabilities are given in Figure \ref{fig:tp} for some $1\geq q>h\geq 0$.
Also, the cost functions and the constraint cost functions are as follows:
\begin{alignat*}{1}
&c(0,0)=c(1,0)=0,\quad c(0,1)=c(1,1)=c_2,\\
&d(0,1)=d(0,0)=0,\quad d(1,0)=d(1,1)=c_1\in(0,1).
\end{alignat*}
The terminal costs are zero. The one-step conditional risk measures is the mean semi-deviation (see equation (\ref{MD_risk})) with fixed $\lambda\in [0, 1]$ and $p\in[1,\infty)$. We wish to solve problem $\mathcal{OPT}$ for this example.

Note that, for any $\lambda$ and $p$, function
\begin{alignat*}{1}
f(x):=\lambda x(1-x)^{1/p}+(1-x)
\end{alignat*} 
is a non-increasing function in $x\in[0,1]$. Therefore, $f(q)\leq f(p)\leq f(0)$.
At stage $k=2$, $V_2(1,r_2)=V_2(0,r_2)=0$, and $\Phi_2(1)=\Phi_2(0)=\{0\}$. At stage $k=1$,
\begin{alignat*}{1}
&V_1(0,r_1)=\left\{\begin{array}{ll}
0&\!\textrm{if }r_1\geq 0,\\
\bar{C}&\!\textrm{else}.
\end{array}\right.\\
&V_1(1,r_1)=\left\{\begin{array}{ll}
0&\!\textrm{if }r_1\geq c_1,\\
\bar{C}&\!\textrm{else}.
\end{array}\right.\nonumber 
\end{alignat*}
Also, $\Phi_1(0)=[0,\infty)$ and $\Phi_1(1)=[c_1,\infty)$. At stage $k=0$, define $K^{(x)}:=f(x)c_1$ (hence $K^{(0)}=c_1$) and
\begin{alignat*}{1}
&E_x(r'(0),r'(1)):=r'(0)x+r'(1)(1-x)\\
&M_x(r'(0),r'(1)):=\left(\small\begin{array}{l}
(1-x)[r'(1)-E_x(r'(0),r'(1))]_+^p\\
+x[r'(0)-E_x(r'(0),r'(1))]_+^p
\end{array}\right)^{1/p};
\end{alignat*}
hence, $E_0(r'(0),r'(1))=r'(1)$ and $M_0(r'(0),r'(1))=0$. Then, we can write
\begin{alignat*}{1}
\begin{split}
F_0(0,r_0)=&\emptyset\quad\textrm{if }r_0<K^{(q)}\\
F_0(0,r_0)=&\{(1,r'):r'(0)\in[0,\infty),r'(1)\in[c_1,\infty), \\
&E_q(r'(0),r'(1))+\lambda M_q(r'(0),r'(1))\leq r_0\}\\
&\textrm{if }K^{(q)}\leq r_0<K^{(h)}\\
F_0(0,r_0)=&\{(1,r'):r'(0)\in[0,\infty),r'(1)\in[c_1,\infty), \\
&E_q(r'(0),r'(1))+\lambda M_q(r'(0),r'(1))\leq r_0\}\\
&\bigcup\{(0,r'):r'(0)\in[0,\infty),r'(1)\in[c_1,\infty), \\
&\qquad E_h(r'(0),r'(1))+\lambda M_h(r'(0),r'(1))\leq r_0\}\\
&\textrm{if } r_0\geq K^{(h)}\\
\end{split}
\end{alignat*}
\begin{alignat*}{1}
\begin{split}
F_0(1,r_0)=&\emptyset\quad\textrm{if }r_0<c_1+K^{(q)}\\
F_0(1,r_0)=&\{(1,r'):r'(0)\in[0,\infty),r'(1)\in[c_1,\infty), \\
&c_1+E_q(r'(0),r'(1))+\lambda M_q(r'(0),r'(1))\leq r_0\}\\
&\textrm{if }c_1+K^{(q)}\leq r_0<c_1+K^{(0)}\\
F_0(1,r_0)=&\{(1,r'):r'(0)\in[0,\infty),r'(1)\in[c_1,\infty), \\
&c_1+E_q(r'(0),r'(1))+\lambda M_q(r'(0),r'(1))\leq r_0\}\\
&\bigcup\{(0,r'):r'(0)\in[0,\infty),r'(1)\in[c_1,\infty), \\
&\qquad c_1+r'(1)\leq r_0\}\\
&\textrm{if } r_0\geq c_1+K^{(0)}\\
\end{split}
\end{alignat*}
As a consequence,
\begin{alignat*}{1}
&V_0(1,r_0)=\left\{\begin{array}{ll}
\bar{C}&\textrm{if }r_0<c_1+K^{(q)}\\
c_2&\textrm{if }c_1+K^{(q)}\leq r_0<c_1+K^{(0)}\\
0&\textrm{if } r_0\geq K^{(0)}\\
\end{array}\right.\nonumber\\
&V_0(0,r_0)=\left\{\begin{array}{ll}
\bar{C}&\textrm{if }r_0<K^{(q)}\\
c_2&\textrm{if }K^{(q)}\leq r_0<K^{(h)}\\
0&\textrm{if } r_0\geq K^{(h)}\\
\end{array}\right.\nonumber\\
\end{alignat*}
Therefore, for $V_0(1,c_1+K^{(q)})$, the infimum of the Bellman's equation is attained with $u=1$, $r'(0)=0$, $r'(1)=c_1$. For $V_0(0,K^{(h)})$, the infimum of the Bellman's equation is attained with $u=0$, $r'(0)=0$, $r'(1)=c_1$. Note that, as expected, the value function is a decreasing function with respect to the risk threshold.
\addtolength{\textheight}{-13cm} 
\section{Conclusions}\label{sec:conc}
In this paper we have presented a dynamic programing approach to stochastic optimal control problems with dynamic, time-consistent (in particular Markov) risk constraints. We have shown that the optimal cost functions can be computed by value iteration and that the optimal control policies can be constructed recursively. This paper leaves numerous important extensions open for further research. First, it is of interest to study how to carry out the Bellman's equation efficiently; a possible strategy involving convex programming has been briefly discussed. Second, to address problems with large state spaces, we plan to develop approximate dynamic programing algorithms for problem $\mathcal{OPT}$. Third, it is of both theoretical and practical interest to study the relation between stochastic optimal control problems with time-consistent and time-inconsistent constraints, e.g., in terms of the optimal costs. Fourth, we plan to extend our approach to the case with partial observations and an infinite horizon. Finally, we plan to apply our approach to real settings, e.g., to the architectural analysis of planetary missions or to the risk-averse optimization of multi-period investment strategies. 

\bibliographystyle{unsrt} 
\bibliography{ref_dyn_pro}  

\end{document}